\documentclass[
final
]{dmtcs-episciences}

\usepackage[utf8]{inputenc}
\usepackage{subfigure}

\usepackage{mathtools}
\usepackage{amssymb}

\newcommand{\NDPF}{\textrm{PF}^{\uparrow}}
\newcommand{\PF}{\textrm{PF}}
\newcommand{\al}{\mathbf{\alpha}}
\newcommand{\be}{\mathbf{\beta}}

\newtheorem{theorem}{Theorem}[section]

\newtheorem{lemma}[theorem]{Lemma}
\newtheorem{example}[theorem]{Example}

\newtheorem{remark}[theorem]{Remark}
\newtheorem{axiom}[theorem]{Axiom}
\newtheorem{definition}[theorem]{Definition}

\usepackage[round]{natbib}

\author[Elder, Harris, Kretschmann, and Mart\'inez Mori]{
    Jennifer Elder\affiliationmark{1}\thanks{
    J.~Elder was partially supported through an AWM Mentoring Travel Grant.
    }
    \and Pamela E. Harris\affiliationmark{2}\thanks{
    P.~E.~Harris was supported through a Karen Uhlenbeck EDGE Fellowship.
    }
    \and Jan Kretschmann\affiliationmark{2}
    \\ \and J.~Carlos Mart\'inez Mori\affiliationmark{3}\thanks{
    J.~C. Mart\'inez Mori is supported by Schmidt Science Fellows, in partnership with the Rhodes Trust.
    Part of this research was performed while J.~C. Mart\'inez Mori was visiting the Mathematical Sciences Research Institute (MSRI), now becoming the Simons Laufer Mathematical Sciences Institute (SLMath), which is supported by NSF Grant No. DMS-1928930.
    }
}
\title[Cost-sharing in Parking Games]{Cost-sharing in Parking Games}
\affiliation{
  Department of Computer Science, Mathematics and Physics, Missouri Western State University, St. Joseph, MO, USA \\
  Department of Mathematical Sciences, University of Wisconsin-Milwaukee, Milwaukee, WI, USA \\
  H. Milton Stewart School of Industrial and Systems Engineering, Georgia Institute of Technology, Atlanta, GA, USA
}
\begin{document}

\publicationdata{vol. 26:3}{2024}{5}{10.46298/dmtcs.13113}{2024-02-25; 2024-02-25; 2024-09-04}{2024-09-07}

\maketitle
\begin{abstract}
In this paper, we study the total displacement statistic of parking functions from the perspective of cooperative game theory.
We introduce \emph{parking games}, which are coalitional cost-sharing games in characteristic function form derived from the total displacement statistic.
We show that parking games are supermodular cost-sharing games, indicating that cooperation is difficult (i.e., their \emph{core} is empty).
Next, we study their \emph{Shapley value}, which formalizes a notion of ``fair'' cost-sharing and amounts to charging each car for its expected marginal displacement under a random arrival order.
Our main contribution is a polynomial-time algorithm to compute the Shapley value of parking games, in contrast with known hardness results on computing the Shapley value of arbitrary games.
The algorithm leverages the permutation-invariance of total displacement, combinatorial enumeration, and dynamic programming.
We conclude with open questions around an alternative solution concept for supermodular cost-sharing games and connections to other areas in combinatorics.
\end{abstract}

\section{Introduction}
\label{sec: introduction}

Consider a one-way street with $n \in \mathbb{N} \coloneqq \{1, 2, \ldots\}$ numbered parking spots.
A sequence of $n$ cars arrive one at a time, each with a preferred spot.
Upon the arrival of car $i \in [n] \coloneqq \{1, 2, \ldots, n\}$, it drives to its preferred spot $a_i \in [n]$.
If spot $a_i$ is unoccupied, it parks there.
Otherwise, it is \emph{displaced} down the one-way street until it finds the first unoccupied spot in which to park, if one exists.
If no such spot exists, the car is unable to park and the parking process fails.
Let $\al = (a_1, a_2, \ldots, a_n) \in [n]^n$ be the $n$-tuple encoding the parking preference of every car.
If all cars are able to park, then  $\al$ is a \emph{parking function} of length $n$.

Parking functions were independently introduced by~\cite{konheim1966occupancy} in their study of hashing functions and by~\cite{pyke1959supremum} in his study of Poisson processes.
Since then, parking functions have become classical objects in combinatorics, displaying rich mathematical structure of their own, as well as diverse connections to other research areas; including hashing~\cite{knuth1998linear}, hyperplane arrangements~\cite{stanley1996hyperplane}, noncrossing partitions~\cite{stanley1997parking}, spanning trees~\cite{kreweras1980famille}, Dyck paths~\cite{armstrong2016rational}, polyhedral combinatorics~\cite{amanbayeva2021convex,stanley2002polytope}, sandpile groups~\cite{cori2000sandpile}, the Bruhat order~\cite{elder2023boolean}, Brownian motion~\cite{diaconis2017probabilizing}, and sorting~\cite{harris2023lucky}, to name just a few.
Refer to \cite{martinezmori2024what,carlson2021parking} for expository introductions to parking functions and their many variants, and to \cite{yan2015parking} for a survey of results.

In this paper, we study the ``fair'' distribution of parking costs.
Specifically, we take the \emph{total displacement} collectively incurred by all cars as the basis for the cost of parking.
As a motivating example, consider the all-ones parking function $(1, 1, \ldots, 1) \in \PF_n$.
In this case, car $1$ is lucky and parks in spot $1$ without incurring any displacement.
However, car $2$ is not as lucky, and it is displaced one unit before parking in spot $2$.
Similarly, car $i$ is displaced $i - 1$ units before parking in spot $i$.
Therefore, the total displacement of $(1, 1, \ldots, 1)$ is 
\begin{equation*}
    \sum_{i=1}^n (i-1) = \frac{n(n-1)}{2}.
\end{equation*}
Certainly, one can recover the total displacement by charging each car for the displacement it incurs: that is, charging $i - 1$ units to each car $i$.
However, this seems rather unfair given that all cars have the same preference; it just so happens that certain cars arrive before others (and the arrival order might be beyond the cars' control).
Therefore, in this case, it seems only fair to charge $(n-1)/2$ units to each car $i$, which again recovers the total displacement.
Now, suppose some car changed its preference from spot $1$ to spot $2$.
This would ever so slightly alleviate the parking demand around spot $1$, and in fact the total displacement would decrease by one unit.
How should this car be ``fairly'' compensated for its more favorable preference?

\subsection{Summary of Results}
We answer this question through the lens of cooperative game theory (refer to \cite{peleg2007introduction} for a comprehensive treatment of this area, and to \cite{myerson1991game} for its broader context in game theory).
Our contributions are as follows.

We first note that the total displacement of a parking function is invariant under the action of the permutation group, even when there are more spots than cars (Theorem~\ref{theorem: invariant}).
In turn, this allows us to define \emph{parking games} as a class of (transferable utility) cooperative games in characteristic function form (Definition~\ref{definition: parking game}).

We then show that parking games are supermodular cost-sharing games (Lemma~\ref{lemma: supermodularity}).
As such, their \emph{core}~\cite{shapley19555markets,gillies1959solutions} is typically empty, indicating that cooperation is difficult.
Supermodularity arises whenever costs are exacerbated by ``congestion'' effects, such as in scheduling~\cite{goemans2002single,queyranne1993structure,schulz2010sharing,schulz2013approximating}.
In much the same way, parking games are supermodular because a car that arrives at an already busy street tends to be displaced significantly before it finds an unoccupied spot.

Next, we adopt the \emph{Shapley value}~\cite{shapley1953value} as a notion of ``fair'' cost-sharing.
In parking games, the Shapley value amounts to charging each car for its \emph{expected marginal displacement} assuming its arrival order is determined uniformly at random.
However, a simple computation of this quantity requires exponential time.
Therefore, our main contribution is a polynomial-time algorithm to compute the Shapley value of parking games (Theorem~\ref{theorem: main}).
Our algorithm leverages the permutation-invariance of total displacement, combinatorial enumeration, and dynamic programming.
We contrast this positive result for the special case of parking games with known hardness results on computing the Shapley value for arbitrary cooperative games~\cite{deng1994complexity,faigle1992shapley}.

Finally, we note that unlike the \emph{scheduling games} studied by \cite{schulz2010sharing}, which are similar to parking games in that they are supermodular and their Shapley value can be computed in polynomial time, the Shapley value of parking games is not a \emph{least core} allocation (i.e., a cost-share distribution that minimizes the worst-case dissatisfaction from cooperation) (Lemma~\ref{lemma: not lca}).
We conclude with open questions around the least core and \emph{least core value} of parking games.

\subsection{Organization}
The remainder of this paper is organized as follows.
In Section~\ref{sec: background}, we present some necessary background on parking functions and cooperative game theory.
In Section~\ref{sec: parking games}, we define parking games, establish supermodularity, and introduce our algorithm.
We conclude in Section~\ref{sec: conclusion} with some additional properties and open questions.

\section{Background}
\label{sec: background}

\subsection*{Notation}

We briefly outline some notational conventions.
All tuples considered have positive integer entries and are denoted in boldface, as in $\be = (b_1, b_2, \ldots, b_n)\in\mathbb{N}^n$.
Throughout, weakly increasing tuples are furthermore decorated with an apostrophe, as in $\be' = (b_1', b_2', \ldots, b_n')$. 
If $\be = (b_1, b_2, \ldots, b_n)$ is an $n$-tuple and $i \in [n]$, then $|\be| = n$ denotes its size and
\begin{equation*}
    \be_{\hat{i}} = (b_1, \ldots, b_{i-1},b_{i+1}, \ldots, b_n)
\end{equation*}
denotes the $(n-1)$-tuple obtained from $\be$ upon the removal of its $i$th entry.
For any $n \in \mathbb{N}$, let $\mathfrak{S}_n$ denote the symmetric group over $[n]$.
If $\pi \in \mathfrak{S}_n$, then $\pi^{-1}$ denotes its inverse.

\subsection{Parking Functions}
\label{sec: parking functions}

Let $\PF_n \subseteq [n]^n$ denote the set of parking functions of length $n$.
Similarly, let $\NDPF_n \subseteq \PF_n$ denote the set of weakly increasing parking functions of length $n$.
\cite{konheim1966occupancy} showed that $|\PF_n| = (n+1)^{n-1}$ (OEIS \href{https://oeis.org/A000272}{A000272}).
We refer the reader to \cite{riordan1969ballots} for an elegant proof credited to Pollak.
\cite{stanley1997parking} showed, through a connection to noncrossing partitions, that $|\NDPF_n| = C_n$, where $C_n = \frac{1}{n+1}\binom{2n}{n}$ is the $n$th Catalan number (OEIS \href{https://oeis.org/A000108}{A000108}).

Let $\al = (a_1, a_2, \ldots, a_n) \in [n]^n$ and $\al' = (a_1', a_2', \ldots, a_n')$ be its weakly increasing rearrangement.
It is well-known (for instance, refer to \cite{yan2015parking}) that $\al \in \PF_n$ if and only if $a_i' \leq i$ for all $i \in [n]$.
As a consequence, $\PF_n$ consists of the orbits of the elements of $\NDPF_n$ under the action of $\mathfrak{S}_n$ (which permutes the subscripts).
In particular, $\PF_n$ is closed under the action of permutations.
Later in this work, we use the following definition.
For any fixed $\al \in \PF_n$, let $r: [n] \rightarrow [n]$ be a bijective rank function that maps each index $i$ in $\al$ to its index $r(i)$ in $\al'$ with ties broken arbitrarily, so that $a_i = a_{r(i)}'$.
Given any subset $S \subseteq [n]$, we denote $r(S) = \{r(i) : i \in S\}$.

\cite{konheim1966occupancy} also considered parking functions where there are more parking spots than cars.
Let $n, m \in \mathbb{N}$ with $m \geq n$.
Suppose there are $m$ numbered parking spots on the one-way street and $n$ cars arrive one at a time, each with a preferred spot in $[m]$.
Let $\al = (a_1, a_2, \ldots, a_n) \in [m]^n$ be the $n$-tuple encoding the parking preferences of every car.
Each car $i$ follows the same parking rule as before: it parks in the first unoccupied spot at or past its preferred spot $a_i \in [m]$, if one exists.
Let $\PF_{n, m} \subseteq [m]^n$ denote the set of preference tuples with $n$ cars and $m$ spots under which all cars are able to park; we refer to these as $(n,m)$-parking functions.
\cite[Lemma~2]{konheim1966occupancy} showed that $|\PF_{n, m}| = (m+1)^{n-1}(m+1-n)$.
Similarly, let $\NDPF_{n,m} \subseteq \PF_{n,m}$ denote the set of weakly increasing $(n,m)$-parking functions.

\subsection{Shapley Value and Core}
\label{sec: shapley value and core}

Let $n \in \mathbb{N}$.
A (transferable utility) \emph{coalitional cost-sharing game} over a set of players $[n]$ is specified by a \emph{characteristic function} $c: 2^{[n]} \rightarrow \mathbb{R}$ satisfying $c(\emptyset) = 0$.
Here, $c(S)$ is the collective cost incurred by the members of the coalition $S \subseteq [n]$ should they act in unison.
The set $[n]$ is referred to as the \emph{grand coalition}.
In simpler terms, the assumption of transferable utility is the existence of a tradable commodity (e.g., money), assumed to be subject to identical valuation from each player, so that the single number $c(S)$ suffices to capture the feasible cost-share allocations for the members of a coalition $S \subseteq [n]$.
A \emph{solution concept} is a function $\phi$ that associates with each game $c$ a subset of player cost-share allocations.
This theory assumes coalitions can reach binding agreements.
Even in scenarios where cooperation may be undesirable for individual agents, as \cite{schulz2010sharing} put it, it might be in the interest of an external party such as a government authority to encourage/enforce cooperation (e.g., through contracts, monetary penalties) to alleviate the negative of externalities of failure to cooperate.

\cite{shapley1953value} (a 2012 Nobel laureate for contributions to game theory~\cite{roth2016lloyd}) considered solution concepts in cost-sharing games through an axiomatic approach.
In particular, he posited the following as desirable properties of a solution concept.
\begin{axiom}[Efficiency]
\label{axiom: efficiency}
$\sum_{i=1}^n \phi_i(c) = c([n])$.
\end{axiom}
\begin{axiom}[Nullity]
\label{axiom: nullity}
If $c(S + \{i\}) = c(S)$ for every $S \subseteq [n] \setminus \{i\}$, then $\phi_i(c) = 0$.   
\end{axiom}
\begin{axiom}[Symmetry]
\label{axiom: symmetry}
If $c(S \cup \{i\}) = c(S + \{j\})$ for every $S \subseteq [n] \setminus \{i, j\}$, then $\phi_i(c) = \phi_j(c)$.
\end{axiom}
\begin{axiom}[Additivity]
\label{axiom: additivity}
If $c'$ is another characteristic function over the set of players $[n]$, then
\begin{equation*}
  \phi_i\left(c + c' \right) = \phi_i(c) + \phi_i(c')  
\end{equation*}
for each $i \in [n]$.
\end{axiom}
Axiom~\ref{axiom: efficiency} states that the sum of cost-shares recover the total cost incurred by the grand coalition.
Axiom~\ref{axiom: nullity} states that if player $i$ never increases the cost of cooperation, in the sense that the marginal cost $c(S + \{i\}) - c(S)$ is zero for all coalitions $S \subseteq [n] \setminus \{i\}$ that player $i$ could join, then player $i$ is correspondingly charged nothing.
Axiom~\ref{axiom: symmetry} states that if players $i$ and $j$ are equivalent with respect to $c$, in the sense that $c(S + \{i\}) = c(S + \{j\})$ for all coalitions $S \subseteq [n] \setminus \{i, j\}$ that players $i$ and $j$ could join (so that all of such coalitions are ambivalent between which of the two players joins them), then players $i$ and $j$ are correspondingly charged the same.
Lastly, Axiom~\ref{axiom: additivity} captures the idea that, if the players participate in a game that combines two independent (possibly completely different) games, then the outcome of one the games does not affect the outcome of the other.

Shapley showed that there is a unique solution concept satisfying Axioms~\ref{axiom: efficiency}-\ref{axiom: additivity}.
This solution concept is now known as the \emph{Shapley value}.
\begin{theorem}[\cite{shapley1953value}]
\label{theorem: shapley}
Let $\mathcal{C}$ be the set of characteristic functions over the set of players $[n]$.
The function $\phi: \mathcal{C} \rightarrow \mathbb{R}^n$ given by
\begin{equation}
\label{eq: shapley}
    \phi_i(c) = \frac{1}{n!} \sum_{\pi \in \mathfrak{S}_n} c \left(\{j \in [n]: \pi(j) \leq \pi(i)\}\right) - c \left(\{j \in [n]: \pi(j) < \pi(i)\}\right)
\end{equation}
for each $i \in [n]$ is the unique function satisfying Axioms~\ref{axiom: efficiency}-\ref{axiom: additivity}.
\end{theorem}
In other words, in the Shapley value, each player is charged their \emph{expected marginal cost} assuming the order in which they join the grand coalition is determined uniformly at random.
Refer to \cite[Chapter~9.4]{myerson1991game} for an approachable presentation of Axioms~\ref{axiom: efficiency}-\ref{axiom: additivity} and of Theorem~\ref{theorem: shapley}.

The Shapley value can be readily computed with exponentially-many oracle calls to the characteristic function (i.e., querying the value $c(S)$ for a given $S \subseteq [n]$), so it is natural to ask whether this many evaluation is generally necessary.
\cite[Theorem~3]{faigle1992shapley} showed that any algorithm that computes the Shapley value for \emph{arbitrary} characteristic functions requires exponentially-many calls.
Moreover, \cite[Theorem~9]{deng1994complexity} showed that, in general, computing the Shapley value is $\#P$-complete.
However, there are nontrivial special cases for which polynomial-time algorithms are known (e.g.,~\cite[Theorem~1]{deng1994complexity}).
In this work, we show that parking games are another such special case.

\cite{gillies1959solutions,shapley19555markets} introduced another solution concept, known as the \emph{core}, with desirable ``stability'' properties.
The core of a cooperative game $c$ is the set
\begin{equation*}
    \left\{\phi \in \mathbb{R}^n : \sum_{i=1}^n \phi_i = c([n]), \  \sum_{i \in S} \phi_i \leq c(S) \text{ for all } S \subseteq [n] \right\}.
\end{equation*}
Intuitively, it is the set of cost-share allocations that are simultaneously efficient (Axiom~\ref{axiom: efficiency}) and robust against coalitional defections.
In other words, no coalition is charged more than the cost that they would incur by themselves.

The Bondareva-Shapley theorem~\cite{bondareva1963some, shapley1967balanced} (refer also to~\cite[Chapter~3.1]{peleg2007introduction}) provides necessary and sufficient conditions for the non-emptiness of the core.
In particular, it implies that if $c$ is \emph{submodular}, then the core is non-empty.
Note that $c$ is submodular if 
\begin{equation*}
    c(S \cup \{i\}) - c(S) \geq c(T \cup \{i\}) - c(T) 
\end{equation*}
for all $i \in [n]$ and all $S \subseteq T \subseteq [n] \setminus \{i\}$.
Conversely, it is \emph{supermodular} if
\begin{equation*}
    c(S \cup \{i\}) - c(S) \leq c(T \cup \{i\}) - c(T) 
\end{equation*}
for all $i \in [n]$ and all $S \subseteq T \subseteq [n] \setminus \{i\}$.
It is \emph{modular} if it is simultaneously submodular and supermodular.
Intuitively, submodularity and supermodularity capture the notions of decreasing and increasing marginal costs, respectively.
It can be verified that the core of a supermodular game is empty unless it is modular, indicating that cooperation is difficult.

\section{Parking Games}
\label{sec: parking games}

\subsection{Displacement, Total Displacement, and Arrival Order}
\label{sec: displacement, total displacement, and arrival order}

Fix any $\al = (a_1, a_2, \ldots, a_n) \in \PF_n$.
Then, by definition, car $i$ with preference $a_i \in [n]$ parks in some spot $p_i \in [n]$ with $p_i \geq a_i$.
Therefore, the displacement incurred by car $i$ upon its arrival under $\al$ is given by
\begin{equation}
\label{eq: individual displacement}
    p_i - a_i \geq 0.
\end{equation}
Now, let $d: \PF_n \rightarrow \mathbb{N}$ be the \emph{total displacement} function, where $d(\al)$ is the total displacement incurred by all cars upon their arrival with preferences $\al \in \PF_n$. 
Using Equation~\eqref{eq: individual displacement}, for any fixed $\al \in \PF_n$ we have
\begin{equation}
\label{eq: disp}
    d(\al)  \coloneqq \sum_{i=1}^n p_i - a_i.
\end{equation}
For example, given a weakly increasing parking function $\al' = (a_1', a_2', \ldots, a_n') \in \NDPF_n$, the cars $1, 2, \ldots, n$ park in the order $1, 2, \ldots, n$.
Therefore, in this case we have $d(\al') = \sum_{i=1}^n (i-a_i')$.

We show that the total displacement statistic is invariant under rearrangement of the entries of a parking function.
This result is standard (e.g., it is stated in \cite[Chapter 13.2.2]{yan2015parking}), but we formally state and prove it for completeness and later use.
\begin{lemma}
\label{lemma: d(x) = d(pi(x))}
If $\al = (a_1, a_2, \ldots, a_n) \in \PF_n$ and $\pi \in \mathfrak{S}_n$ acts on $\al$ by permuting its subscripts, i.e. $\pi(\al)=(a_{\pi^{-1}(1)}, a_{\pi^{-1}(2)}, \ldots, a_{\pi^{-1}(n)})$, then $d(\al)=d(\pi(\al))$.
\end{lemma}
\begin{proof}
Note that under $\al$, car $i$ with preference $a_i \in [n]$ parks in some spot $p_i \in [n]$. 
Since $\al \in \PF_n$, this implies $\{p_1, p_2, \ldots, p_n\} = [n]$ and we have
\begin{equation*}
   d(\al) 
   = \sum_{i=1}^{n} \left(p_i - a_i\right)
   = \sum_{i=1}^{n} i - \sum_{i=1}^{n} a_i
   = \frac{n(n+1)}{2} - \sum_{i=1}^n a_i. 
\end{equation*}
Similarly, note that under $\pi(\al)$, car $i$ with preference $a_{\pi^{-1}(i)}$ parks in some spot $q_i \in [n]$.
Since $\pi(\al) \in \PF_n$, this implies $\{q_1, q_2, \ldots, q_n\} = [n]$ and we have
\begin{equation*}
    d(\pi(\al))
   = \sum_{i=1}^{n} \left(q_i - a_{\pi^{-1}(i)}\right)
   = \sum_{i=1}^{n} i - \sum_{i=1}^{n} a_{\pi^{-1}(i)}
   = \frac{n(n+1)}{2} - \sum_{i=1}^{n} a_{\pi^{-1}(i)}.
\end{equation*}
Lastly, note that 
\begin{equation*}
    \sum_{i=1}^n a_i = \sum_{i=1}^{n} a_{\pi^{-1}(i)}
\end{equation*}
since $\pi$ is a bijection from $[n]$ to $[n]$.
\end{proof}

We now generalize this result to the case of $(n,m)$-parking functions.
For $n, m \in \mathbb{N}$ with $m \geq n$, let $d: \PF_{n,m} \rightarrow \mathbb{N}$ be the total displacement function, where $d(\al)$ is the total displacement incurred by all cars upon their arrival with preferences $\al \in \PF_{n,m}$.

\begin{theorem}
\label{theorem: invariant}
If $\al = (a_1, a_2, \ldots, a_n) \in \PF_{n,m}$ and $\pi\in\mathfrak{S}_n$ acts on $\al$ by permuting its subscripts, i.e. $\pi(\al) = (a_{\pi^{-1}(1)}, a_{\pi^{-1}(2)}, \ldots, a_{\pi^{-1}(n)})$, then $d(\al)=d(\pi(\al))$.
\end{theorem}
\begin{proof}
Upon replicating the proof technique from Lemma~\ref{lemma: d(x) = d(pi(x))}, it suffices to show that the subset of occupied spots under $\al$ is the same as the subset of occupied spots under $\pi(\al)$ (this is immediate in Lemma~\ref{lemma: d(x) = d(pi(x))} since it assumes $\al, \pi(\al) \in \PF_n$).

Suppose that under $\al$, the cars park in spots in $S \subseteq [m]$ whereas under $\pi(\al)$, the cars park in spots in $T \subseteq [m]$.
Note that $|S| = |T| = n$.
If $S = T$, then the result follows.
Suppose by way of contradiction that $S\neq T$, and let $p$ be the smallest value in $(S \setminus T) \cup (T \setminus S)$.
Without loss of generality, suppose $p \in S \setminus T$ (otherwise we swap the roles in the following argument).
Since $p \in S$, there are $k$ cars with preference less than or equal to $p$ in $\alpha$ for some $1 \leq k \leq n$.
Conversely, since $p \notin T$ and $p$ is the smallest value in $S \setminus T$, there are $k-1$ cars with preference less than or equal to $p$ in $\pi(\alpha)$.
This is a contradiction, for $\al$ and $\pi(\al)$ are equal as multisets.
\end{proof}

Given Theorem~\ref{theorem: invariant}, we now define parking games as coalitional games in characteristic function form, where each car is treated as a player.
\begin{definition}[Parking Game]
\label{definition: parking game}
Let $\al = (a_1, a_2, \ldots, a_n) \in \PF_n$ be a parking function.
The parking game of $\al$ is given by the characteristic function $c_\al: 2^{[n]} \rightarrow \mathbb{N}$, where
\begin{equation}
    c_\al(S) = d((a_i: i \in S))
\end{equation}
for each $S \subseteq [n]$.
\end{definition}
In other words, $c_\al(S)$ is the total displacement of the $(|S|,n)$-parking function $(a_i : i \in S) \in \PF_{|S|, n}$.
Note that this definition relies on Theorem~\ref{theorem: invariant} to treat $c_\al$ as a set function that is independent of the arrival order of the cars in $S$.
As noted in the following remark, the parking game of a parking function can be obtained through the parking game of its weakly increasing rearrangement.
\begin{remark} 
Let $\al = (a_1, a_2, \ldots, a_n) \in \PF_n$ be a parking function and $\al' = (a_1', a_2', \ldots, a_n') \in \NDPF_n$ be its weakly increasing rearrangement.
Then, by Theorem~\ref{theorem: invariant}, it follows that
\begin{equation*}
    c_\al(S) = d((a_i: i \in S)) = d((a_i': i \in r(S))) = c_{\al'}(r(S))
\end{equation*}
for each $S \subseteq [n]$.
\end{remark}

\subsection{Supermodularity}
\label{sec: supermodularity}

We now show that parking games are supermodular cost-sharing games.
\begin{lemma}
\label{lemma: supermodularity}
Let $\al = (a_1, a_2, \ldots, a_n) \in \PF_n$ be a parking function.
Then, $c_\al$ is supermodular, i.e.,
\begin{equation*}
    c_\al(S \cup \{i\}) - c_\al(S) \leq c_\al(T \cup \{i\}) - c_\al(T) 
\end{equation*}
for all $i \in [n]$ and $S \subseteq T \subseteq [n] \setminus \{i\}$.
\end{lemma}
\begin{proof}
Fix any $i \in [n]$ and $S \subseteq T \subseteq [n] \setminus \{i\}$.
Consider the arrival of car $i$ immediately after the arrival of the cars in $S$.
If spot $a_i$ is empty, then $c_\al(S \cup \{i\}) = c_\al(S)$ and the inequality is verified.
Conversely, suppose there is a sequence of contiguously occupied spots starting at some spot $ s$ satisfying $1 \leq s \leq a_i$ and ending at some spot $ t$ satisfying $a_i \leq t \leq n -1$, implying the displacement of car $i$ is $t - a_i + 1$.
Now, similarly consider the arrival of car $i$ immediately after the arrival of the cars in $T \supseteq S$.
Then, there is again a sequence of contiguously occupied spots, except this time starting at some spot $s'$ satisfying $1 \leq s' \leq s$ and ending at some spot $t'$ satisfying $t \leq t' \leq n -1$, implying the displacement of car $i$ is $t' - a_i + 1 \geq t - a_i + 1$.
\end{proof}

Supermodularity captures the notion of increasing marginal costs, and it arises whenever costs are exacerbated by ``congestion'' effects, such as in scheduling~\cite{goemans2002single,queyranne1993structure,schulz2010sharing,schulz2013approximating}.
Intuitively, parking games are supermodular because a car that arrives at an already busy street tends to displace significantly before it finds an unoccupied spot.
\cite[Theorem~1]{schulz2010sharing} showed that supermodular games are in some sense common: the problem of minimizing a non-negative linear function over a supermodular polyhedron, which arises often in combinatorial optimization, has a supermodular objective value.
In effect, parking games are a special family within the broader class of supermodular games.

It follows from Lemma~\ref{lemma: supermodularity} that the core of parking games is typically empty (with the exception of modular cases in which $\al \in \mathfrak{S}_n$).
For example, this can be verified from the fact that each car in isolation incurs no displacement.

\subsection{Expected Marginal Displacement}
\label{sec: expected marginal displacement}

Let $\al \in \PF_n$ be a parking function and $c_\al: 2^{[n]} \rightarrow \mathbb{N}$ be its parking game. 
In this section, we present a polynomial-time algorithm to compute the Shapley value of $c_\al$.

In this setting, based on Equation~\eqref{eq: shapley}, each car $i$ is assigned a cost-share of
\begin{equation}
\label{eq: expected marginal displacement}
    \phi_i(c_\al) 
    = \frac{1}{n!} \sum_{\pi \in \mathfrak{S}_n} c_\al \left(\{j \in [n]: \pi(j) \leq \pi(i)\}\right) - c_\al \left(\{j \in [n]: \pi(j) < \pi(i)\}\right).
\end{equation}
Note that given any arrival order $\pi \in \mathfrak{S}_n$, the difference 
\begin{equation*}
    c_\al \left(\{j \in [n]: \pi(j) \leq \pi(i)\}\right) - c_\al \left(\{j \in [n]: \pi(j) < \pi(i)\}\right)
\end{equation*}
corresponds to the displacement of car $i$ upon its arrival in the order given by $\pi$.
In effect, Equation~\eqref{eq: expected marginal displacement} is a formula for the \emph{expected marginal displacement} of car $i$ assuming its arrival order is determined uniformly at random.
We leverage this interpretation in what follows.

\begin{example}
\label{example: ex1}
Let $\alpha = (1, 4, 3, 3, 1, 2, 7) \in \PF_7$.
Based on \eqref{eq: disp} we have that $d(\alpha) = 7$.
Based on \eqref{eq: expected marginal displacement} we have that:
\begin{itemize}
    \item 
    $\phi_1(c_\alpha) = 7896 / 7! = 47/30$.
    \item 
    $\phi_2(c_\alpha) = 2856 / 7! = 17/30$.
    \item 
    $\phi_3(c_\alpha) = 5628 / 7! = 67/60$.
    \item 
    $\phi_4(c_\alpha) = 5628 / 7! = 67/60$.
    \item 
    $\phi_5(c_\alpha) = 7896 / 7! = 47/30$.
    \item 
    $\phi_6(c_\alpha) = 5376 / 7! = 16/15$.
    \item 
    $\phi_7(c_\alpha) = 0 / 7! = 0$.
\end{itemize}
Note that $\sum_{i = 1}^n \phi_i(c_\alpha) = 7$ (this reflects Axiom~\ref{axiom: efficiency}).
Moreover, note that cars with the same preference have the same cost-share allocation, such as cars $3$ and $4$ (this reflects Axiom~\ref{axiom: symmetry}).
Lastly, note that car $7$ is lucky (i.e., parks in its preferred spot) regardless of the arrival order.
Therefore, $\phi_7 = 0$ (this reflects Axiom~\ref{axiom: nullity}).
\end{example}

We first introduce some notation.
Given a tuple $\be$ and an integer $t \in [n]$, let
\begin{equation*}
    \Lambda_t(\be) = |\{b \in \be: b \geq t + 2\}|
\end{equation*}
be the number of entries of $\be$ that are greater than or equal to $t + 2$.
Similarly, given an integer $s \in [n]$, let
\begin{equation*}
    \Gamma_s(\be) = |\{b \in \be: b \leq s - 2\}|
\end{equation*}
be the number of entries of $z$ that are less than or equal to $s - 2$.
Moreover, given a weakly increasing tuple $\be'$, and three nonnegative integers $s, t, k \in \mathbb{N}$ with $s \leq t$, let $\mathcal{Q}(\be', s, t, k)$ denote the number of size $k$ sub-tuples of $\be'$ that, when treated as a preference $k$-tuple, cars park in spots $s, \ldots, t$.
Formally,
\begin{equation*}
    \mathcal{Q}(\be', s, t, k) = |\{(i_1, i_2, \ldots, i_k) : (b_{i_1}' - b_{i_1}' + 1, b_{i_2}' - b_{i_1}' + 1, \ldots, b_{i_k}' - b_{i_1}' + 1) \in \NDPF_{k, t - s + 1}\}|.
\end{equation*}

We obtain the following enumeration.
\begin{lemma}
\label{theorem: count}
Let $\al = (a_1, a_2, \ldots, a_n) \in \PF_n$ be a parking function and $\al' = (a_1', a_2', \ldots, a_n') \in \NDPF_n$ be its weakly increasing rearrangement.
Fix any car $j \in [n]$ and let $i \coloneqq r(j) \in [n]$ be its rank.
Then, Equation~\eqref{eq: expected marginal displacement} for car $j$ is given by
\begin{equation}
\label{eq: formula}
    \phi_j(c_{\al})
    =
    \phi_i(c_{\al'})
    =
    \frac{1}{n!}
    \sum_{s=1}^{a_i'}
    \sum_{t=a_i'}^{n - 1}
    (t - a_i' + 1)
    \mathcal{Q}(\al_{\hat{i}}', s, t, t-s+1)
    \mathcal{R}(\al_{\hat{i}}', s, t),
\end{equation}
where
\begin{equation*}
    \mathcal{R}(\al_{\hat{i}}', s, t)
    =
    \sum_{\lambda=0}^{\Lambda_t(\al_{\hat{i}}')}
    \sum_{\gamma=0}^{\Gamma_s(\al_{\hat{i}}')}
    \binom{\Lambda_t(\al_{\hat{i}}')}{\lambda}
    \mathcal{Q}(\al_{\hat{i}}', 1, s - 2, \gamma)
    (t - s + 1 + \lambda + \gamma)!
    (n - t + s - \lambda - \gamma - 2)!.
\end{equation*}
\end{lemma}
\begin{proof}
Note that the cost-share $\phi_j(c_{\al})$ of car $j$ under $\al$ is equal to the cost-share $\phi_i(c_{\al'})$ of car $i$ under $\al'$.
Therefore, consider the displacement of car $i$ upon its arrival under $\al'$.

Upon its arrival, car $i$ drives to its preferred spot $a_i'$.
If spot $a_i'$ is empty, car $i$ parks there incurring zero displacement.
Conversely, suppose spot $a_i'$ is occupied.
Then, there is a sequence of contiguously occupied spots starting at some spot $s$ satisfying $1 \leq s \leq a_i'$ and ending at some spot $t$ satisfying $a_i' \leq t \leq n - 1$; note that $t \neq n$ since, by assumption, car $i$ is able to park.
Now, fix any pair of possible values for $s$ and $t$.
This implies the following (for otherwise the block does not start at $s$ and end at $t$):
\begin{itemize}
    \item Spot $t + 1$ is empty.
    As a consequence, car $i$ parks in spot $t + 1$ incurring $t - a_i' + 1$ displacement.
    \item If spot $s - 1$ exists, it is empty as well.
\end{itemize}

In turn, the choice of $s$ and $t$ leads to three contiguous segments of parking spots in which cars may park prior to the arrival of car $i$:
\begin{enumerate}
    \item Spots $1, \ldots, s - 2$ (if these exist),
    \item spots $s, \ldots, t$, and
    \item spots $t + 2, \ldots, n$ (if these exist).
\end{enumerate}
Note that unlike spots $s, \ldots, t$, spots $1, \ldots, s - 2$ need not be contiguously occupied.
Similarly, spots $t + 2, \ldots, n$ need not be contiguously occupied.
We now count the number of different subsets of cars that can occupy each of these three segments.
\begin{enumerate}
    \item 
    Since spot $s - 1$ is empty, only those cars with preferences less than or equal to $s - 2$ could possibly park in spots $1, \ldots, s - 2$; there are $\Gamma_s(\al_{\hat{i}})$ such cars.
    Only some number $0 \leq \gamma \leq \Gamma_s(\al_{\hat{i}})$ of cars park in these spots prior to the arrival of car $i$.
    Therefore, for any fixed possible value of $\gamma$, the number of size-$\gamma$ subset of cars that park in this segment is given by $Q(\al_{\hat{i}}', 1, s - 2, \gamma)$.
    \item 
    Since spot $s - 1$ is empty, spots $s, \ldots, t$ are contiguously occupied, and spot $t + 1$ is empty, we need the number of size-$(t - s + 1)$ subsets of cars that park in this segment.
    This is given by $Q(\al_{\hat{i}}', s, t, t - s + 1)$.
    \item 
    Since spot $t + 1$ is empty, only those cars with preferences greater than or equal to $t + 2$ could possibly park in spots $t + 2, \ldots, n$; there are $\Lambda_t(\al_{\hat{i}})$ such cars.
    Only some number $0 \leq \lambda \leq \Lambda_t(\al_{\hat{i}})$ of cars park in these spots prior to the arrival of car $i$.
    Therefore, for any fixed possible value of $\lambda$, the number of size-$\lambda$ subset of cars that park in this segment is given by $Q(\al_{\hat{i}}', t + 2, n, \lambda)$.
    In fact, since $\al' \in \NDPF_n$, any size-$\lambda$ subset of cars with preference greater than or equal to $t + 2$ is able to park in the segment $t + 2, \ldots, n$, and it follows that
    \begin{equation*}
        Q(\al_{\hat{i}}', t + 2, n, \lambda) = \binom{\Lambda_t(\al_{\hat{i}})}{\lambda}.
    \end{equation*}
\end{enumerate} 
Note that in (1) above, the term $Q(\al_{\hat{i}}', 1, s - 2, \gamma)$ does not generally simplify into a binomial because, depending on the preferences, certain cars might displace into spot $s - 1$, which we thus far assume to be empty.

Now, the choice of $s$, $t$, $\lambda$, and $\gamma$ imply $(t - s + 1) + \lambda + \gamma$ cars park prior to the arrival of car $i$, in one of $(t - s + 1 + \lambda + \gamma)!$ different arrival orders.
Similarly, $n - 1 - ((t - s + 1) + \lambda + \gamma) = n - t + s - \lambda - \gamma - 2$ cars arrive after the arrival of car $i$, in one of $(n - t + s - \lambda - \gamma - 2)!$ different arrival orders.

Summing over the possible values for $s$, $t$, $\lambda$, and $\gamma$ yields the sum portion of Equation \ref{eq: formula}:
\[
\sum_{s=1}^{a_i'}
    \sum_{t=a_i'}^{n - 1}
    (t - a_i' + 1)
    \mathcal{Q}(\al_{\hat{i}}', s, t, t-s+1)
    \mathcal{R}(\al_{\hat{i}}', s, t)
\]
where 
\[
\mathcal{R}(\al_{\hat{i}}', s, t)
    =
    \sum_{\lambda=0}^{\Lambda_t(\al_{\hat{i}}')}
    \sum_{\gamma=0}^{\Gamma_s(\al_{\hat{i}}')}
    \binom{\Lambda_t(\al_{\hat{i}}')}{\lambda}
    \mathcal{Q}(\al_{\hat{i}}', 1, s - 2, \gamma)
    (t - s + 1 + \lambda + \gamma)!
    (n - t + s - \lambda - \gamma - 2)!
\]
Finally, there are $n!$ different orders in which all cars could arrive, and the arrival order is realized uniformly at random. 
This completes the proof.
\end{proof}

\begin{example}
We demonstrate an execution of \eqref{eq: formula} using the instance in Example~\ref{example: ex1}.
Let 
\begin{equation*}
    \alpha = (1, 4, 3, 3, 1, 2, 7) \in \PF_7
\end{equation*}
and note that its weakly increasing rearrangement is 
\begin{equation*}
    \alpha' = (1, 1, 2, 3, 3, 4, 7) \in \NDPF_7.  
\end{equation*}
Consider $j = 2$, so that $\alpha_2 = 4$ and $r(j) = 6$.
To compute $\phi_2(\alpha)$ we equivalently compute $\phi_6(\alpha')$.
Here we only show one term of its computation.

Suppose $s = 2$ and $t = 4$ in \eqref{eq: formula}.
Then, upon its arrival, car $6$ is displaced $t - a_6' + 1 = 4 - 4 + 1 = 1$ unit.
Now, by our choice of $s$ and $t$, spots $2, 3, 4$ are contiguously occupied whereas spot $1$ and $5$ are unoccupied.
Note the following:
\begin{itemize}
    \item 
    Neither of cars $1$ or $2$ could have arrived, for otherwise spot $1$ would be occupied.
    In particular, $\Gamma_s(\alpha_{\hat{6}}') = \emptyset$ and we only need to consider $\gamma = 0$.
    \item 
    Car $7$ could have arrived, for in either case spot $5$ would remain unoccupied.
    In particular, $\Lambda_t(\alpha_{\hat{6}}') = \{7\}$ and we only need to consider $\lambda = 0, 1$.
    \item 
    Cars $3, 4, 5$ must have arrived, as this is the only way in which spots $2, 3, 4$ would be contiguously occupied.
    In particular, $\mathcal{Q}(\al_{\hat{6}}', s, t, t-s+1) = 1$.
\end{itemize}
In the case in which $\gamma = 0, \lambda = 0$, a total of $3$ cars arrive prior to car $6$ whereas a total of $3$ cars arrive after car $6$.
In the case in which $\gamma = 0, \lambda = 1$, a total of $4$ cars arrive prior to car $6$ whereas a total of $2$ cars arrive after car $6$.
To summarize, this choice of $s$ and $t$ contributes
\begin{equation*}
    \underbrace{1}_{t - a_6' + 1}
    \underbrace{1}_{\mathcal{Q}(\al_{\hat{6}}', s, t, t-s+1)} \left(\underbrace{1 \cdot 1 \cdot 3! \cdot 3!}_{\gamma = 0, \gamma = 0} \ + \ \underbrace{1 \cdot 1 \cdot 4! \cdot 2!}_{\lambda = 1, \gamma = 0} \right) = 84
\end{equation*}
to the total sum.
\end{example}

Assuming oracle access to $\mathcal{Q}$, Equation~\eqref{eq: formula} can be evaluated in polynomial time in $\al$ for any car $j \in [n]$ (note that each of the summations involves at most $n$ terms).
As a final step, we show that $\mathcal{Q}$ can be evaluated in polynomial time as well.

We first introduce a some additional notation.
Given a tuple $\be$, let $b^* = \min \{b \in \be\}$ be the value of its smallest entry and
\begin{equation*}
    U(\be) = (\max\{b, 1 + b^* \} : b \in \be)
\end{equation*}
be the copy of $\be$ in which all entries with value equal to $b^*$ are increased by one.
For example, if $\be = (3, 3, 4, 4, 5)$, then $b^* = 3$ and $U(\be)=(4, 4, 4, 4, 5)$.
We evaluate $\mathcal{Q}$ using the following recursive relation.

\begin{lemma}
\label{theorem: dynamic}
Let $\be' = (b_1', b_2', \ldots, b_{|\be'|}') \in [w]^{|\be'|}$ be a nonnegative, weakly increasing tuple where $w \in \mathbb{N}$, and let $s, t, k \in \mathbb{N}$ with $s \leq t$.
Then, $\mathcal{Q}(\be', s, t, k)$ satisfies the following recursive relation:
\begin{equation}
\label{eq: recursion}
    \mathcal{Q}(\be', s, t, k)
    =
    \begin{cases}
    1, & \text{if $k = 0$}, \\
    0, & \text{if $k > t - s + 1 \lor k > |\be'|$}, \\
    \mathcal{Q}(\be_{\hat{1}}', s, t, k), & \text{if $b_1' < s$}, \\
    \mathcal{Q}(\be', s + 1, t, k), & \text{if $b_1' > s$}, \\
    \mathcal{Q}(U(\be')_{\hat{1}}, s + 1, t, k - 1) + \mathcal{Q}(\be_{\hat{1}}', s, t, k),  & \text{if $b_1' = s$.}
    \end{cases}
\end{equation}
Using dynamic programming, $\mathcal{Q}(\be', s, t, k)$ can be evaluated in time polynomial in $|\be'|$ and $k$.
\end{lemma}
\begin{proof}
We first consider the base case.
If we attempt to park more cars than there are spots (i.e., $k > t - s + 1$), or more cars than there are preferences (i.e., $k > |\be'|$), the count is zero.
Similarly, if we attempt to park no cars (i.e., $k = 0$), the count is one.

Therefore, suppose $1 \leq k \leq \min\{t - s + 1, |\be'|\}$.
In this case, there are three possibilities, each leading to a distinct recursive call (recall $b_1'$ is the first entry of $\be'$):
\begin{itemize}
    \item 
    If $b_1' < s$, then the first car cannot park in the segment $s, \ldots, t$.
    Therefore, in this case, the count is the same as the count $\mathcal{Q}(\be_{\hat{1}}', s, t, k)$ upon the removal of the first car.
    \item 
    If $b_1' = s$, then there are two possibilities for the first car: it is selected as part of the size-$k$ subset of cars, or it is not.
    In the first sub-case, no subsequent car can park in the spot $s$ that starts the segment $s, \ldots, t$.
    Therefore, in this sub-case, the count is the same as the count $\mathcal{Q}(U(\be')_{\hat{1}}, s + 1, t, k - 1)$ upon increasing the preference of any car that prefers spot $s$ by one, removing the first car, increasing the segment start by one spot, and decreasing the number of cars to select by one.
    In the second sub-case, the count is the same as the count $\mathcal{Q}(\be_{\hat{1}}', s, t, k)$ upon the removal of the first car.
    \item 
    If $b_1' > s$, then the first car cannot park in the spot $s$ that starts the segment $s, \ldots, t$. 
    Therefore, in this case, the count is the same as the count $\mathcal{Q}(\be', s + 1, t, k)$ upon increasing the segment start by one spot.
\end{itemize}
Finally, note that a dynamic programming table of polynomial size can be implemented since its indices require at most $|\be'|$ suffixes of $\be'$, $s,t \leq |\be'|$, and $U$ need only be applied $k$-many times.
\end{proof}

While $w$ does not play any explicit role in the proof of Lemma~\ref{theorem: dynamic}, the values it can possibly take are implicitly restricted by the problem definition which, for any given $n \in \mathbb{N}$, is restricted to $\NDPF_n \subseteq [n]^n$.
We thus obtain our main result.

\begin{theorem}
\label{theorem: main}
There exists a polynomial-time algorithm to compute the Shapley value of parking games.
Namely, the one given by the formulas in Lemma~\ref{theorem: count} and Lemma~\ref{theorem: dynamic}.
\end{theorem}
\begin{proof}
The correctness of the algorithm follows from the proofs of the lemmas.
Its running time follows from Lemma~\ref{theorem: dynamic} and the fact that, in Equation~\eqref{eq: formula}, the recursive relation Equation~\eqref{eq: recursion} is always evaluated passing weakly increasing $\be'$ and $k$ satisfying $|\be'|, k \leq n$.
\end{proof}

\section{Conclusion}
\label{sec: conclusion}

One can envision real-world systems in which displacement is a costly negative externality and, as a result, the parking rate at a given spot is a function of the displacement derived from its popularity.
For example, in transportation settings, this might be because of increased environmental emissions.
The results presented in this work form a methodological basis for operating such a system.

The results in this paper readily extend to $(n,m)$-parking functions by splitting them into small independent ``parking functions,'' in the style underlying 
Theorem~\ref{theorem: invariant}.
Moreover, because of Axiom~\ref{axiom: additivity}, Theorem~\ref{theorem: main} can be applied given any characteristic function that is an affine function of the total displacement.

We conclude with directions for future research.
As noted in Lemma~\ref{lemma: supermodularity}, parking games are supermodular, which indicates that their core is empty (except for modular cases in which $\al \in \mathfrak{S}_n$).
Therefore, the \emph{least core}~\cite{shapley1966quasi,maschler1979geometric} is an appropriate alternative solution concept.
The least core is the set of optimal solutions to 
\begin{equation}
\label{eq: z^*}
    z^* = \min \left\{z : \sum_{i=1}^n \phi_i = c([n]), \ \sum_{i \in S} \phi_i \leq c(S) + z, \text{ for all } S \subseteq [n] \right\},
\end{equation}
and $z^*$ is the \emph{least core value}.
Unlike the core, the least core is always non-empty.
A least core allocation minimizes the worst-case dissatisfaction over all coalitions, and the least core value can be seen as the minimum amount that needs to be charged for defection (e.g., by a governing authority) in order to incentivize cooperation.

\cite{schulz2010sharing, schulz2013approximating} study the least core and least core value of supermodular cost-sharing games.
They show that computing the least core value of an arbitrary supermodular game is strongly NP-hard~\cite[Theorem~2]{schulz2010sharing}.
Interestingly, for \emph{scheduling games}, not only can their Shapley value be computed in polynomial time, but it moreover happens to be a least core allocation~\cite[Theorem~3]{schulz2010sharing}. 
(However, they note that computing the least core value of scheduling games remains weakly NP-hard).
Parking games are similar to scheduling games in that, as we have shown, they are supermodular and their Shapley value can be computed in polynomial time.
However, they are different in that their Shapley value is not a least core allocation.
\begin{lemma}
\label{lemma: not lca}
The Shapley value of parking games is not a least core allocation.
\end{lemma}
\begin{proof}
It suffices to show a counterexample.
Let $\al = (1, 1, 2) \in \PF_3$ and consider its parking game $c_\al$.
We have $c_\al(\emptyset) = c_\al(\{1\}) = c_\al(\{2\}) = c_\al(\{3\}) = c_\al(\{1, 3\}) = c_\al(\{2, 3\})= 0$, $c_\al(\{1, 2\}) = 1$, and $c_\al(\{1, 2, 3\}) = 2$.
Its Shapley value is $\phi_1(c_\al) = \phi_2(c_\al) = 5/6$ and $\phi_3(c_\al) = 2/6$.
However, its least core value is $z^* = 1$ with the least core allocation $\phi_1 = \phi_2 = 1$ and $\phi_3 = 0$.
In particular, for $S = \{1, 3\}$ we have
\begin{equation*}
    \phi_1(c_\al) + \phi_3(c_\al) = 5/6 + 2/6 = 7/6 \nleq 1 = 0 + 1 = c_\al(\{1, 3\}) + z^*.
\end{equation*}
\end{proof}

However, given the positive result on the complexity of computing the Shapley value of parking games, we ask whether a least core allocation or the least core value of parking games can be computed in polynomial time and/or interpreted combinatorially.
First, we ask whether the least core allocation(s) relate to individual parking statistics, including but not limited to those concerning individual displacement.
We also ask about the coalitions $S \subseteq [n]$ for which the least core value $z^*$ is attained in \eqref{eq: z^*} for a particular least core allocation: do such coalitions fully determine $z^*$ in terms of parking statistics that distinguish them from other coalitions?
Similarly, we ask whether the Shapley value of parking games is an \emph{approximate} least core allocation (refer to \cite[Section 2]{schulz2013approximating}).

Finally, given the abundance of connections between parking functions and other combinatorial objects (refer to Section~\ref{sec: introduction} for some examples), future work might consider mappings that preserve the total displacement statistic, in this way defining equivalent cooperative games except with a different combinatorial interpretation.

\acknowledgements
\label{sec:ack}

J.~C. Mart\'inez Mori would like to thank A. Toriello for a helpful discussion.

\nocite{*}
\bibliographystyle{abbrvnat}
\bibliography{bib}

\begin{thebibliography}{34}
\providecommand{\natexlab}[1]{#1}
\providecommand{\url}[1]{\texttt{#1}}
\expandafter\ifx\csname urlstyle\endcsname\relax
  \providecommand{\doi}[1]{doi: #1}\else
  \providecommand{\doi}{doi: \begingroup \urlstyle{rm}\Url}\fi

\bibitem[Amanbayeva and Wang(2022)]{amanbayeva2021convex}
A.~Amanbayeva and D.~Wang.
\newblock The convex hull of parking functions of length {$n$}.
\newblock \emph{Enumer. Comb. Appl.}, 2\penalty0 (2):\penalty0 Paper No. S2R10,
  10, 2022.
\newblock \doi{10.54550/eca2022v2s2r10}.

\bibitem[Armstrong et~al.(2016)Armstrong, Loehr, and
  Warrington]{armstrong2016rational}
D.~Armstrong, N.~A. Loehr, and G.~S. Warrington.
\newblock Rational parking functions and {C}atalan numbers.
\newblock \emph{Ann. Comb.}, 20\penalty0 (1):\penalty0 21--58, 2016.
\newblock \doi{10.1007/s00026-015-0293-6}.

\bibitem[Bondareva(1963)]{bondareva1963some}
O.~N. Bondareva.
\newblock Some applications of the methods of linear programming to the theory
  of cooperative games.
\newblock \emph{Problemy Kibernet.}, 10:\penalty0 119--139, 1963.

\bibitem[Carlson et~al.(2021)Carlson, Christensen, Harris, Jones, and
  Rodriguez]{carlson2021parking}
J.~Carlson, A.~Christensen, P.~E. Harris, Z.~Jones, and A.~R. Rodriguez.
\newblock Parking functions: Choose your own adventure.
\newblock \emph{The College Mathematics Journal}, 52\penalty0 (4):\penalty0
  254--264, 2021.

\bibitem[Cori and Rossin(2000)]{cori2000sandpile}
R.~Cori and D.~Rossin.
\newblock On the sandpile group of dual graphs.
\newblock \emph{European J. Combin.}, 21\penalty0 (4):\penalty0 447--459, 2000.
\newblock \doi{10.1006/eujc.1999.0366}.

\bibitem[Deng and Papadimitriou(1994)]{deng1994complexity}
X.~T. Deng and C.~H. Papadimitriou.
\newblock On the complexity of cooperative solution concepts.
\newblock \emph{Math. Oper. Res.}, 19\penalty0 (2):\penalty0 257--266, 1994.
\newblock \doi{10.1287/moor.19.2.257}.

\bibitem[Diaconis and Hicks(2017)]{diaconis2017probabilizing}
P.~Diaconis and A.~Hicks.
\newblock Probabilizing parking functions.
\newblock \emph{Adv. in Appl. Math.}, 89:\penalty0 125--155, 2017.
\newblock \doi{10.1016/j.aam.2017.05.004}.

\bibitem[Elder et~al.(2023)Elder, Harris, Kretschmann, and
  Mart\'inez~Mori]{elder2023boolean}
J.~Elder, P.~E. Harris, J.~Kretschmann, and J.~C. Mart\'inez~Mori.
\newblock Boolean intervals in the weak order of $\mathfrak{S}_n$.
\newblock \emph{arXiv preprint arXiv:2306.14734}, 2023.

\bibitem[Faigle and Kern(1992)]{faigle1992shapley}
U.~Faigle and W.~Kern.
\newblock The {S}hapley value for cooperative games under precedence
  constraints.
\newblock \emph{Internat. J. Game Theory}, 21\penalty0 (3):\penalty0 249--266,
  1992.
\newblock \doi{10.1007/BF01258278}.

\bibitem[Gillies(1959)]{gillies1959solutions}
D.~B. Gillies.
\newblock Solutions to general non-zero-sum games.
\newblock In \emph{Contributions to the theory of games, {V}ol. {IV}}, volume
  no. 40 of \emph{Ann. of Math. Stud.}, pages 47--85. Princeton Univ. Press,
  Princeton, NJ, 1959.

\bibitem[Goemans et~al.(2002)Goemans, Queyranne, Schulz, Skutella, and
  Wang]{goemans2002single}
M.~X. Goemans, M.~Queyranne, A.~S. Schulz, M.~Skutella, and Y.~Wang.
\newblock Single machine scheduling with release dates.
\newblock \emph{SIAM J. Discrete Math.}, 15\penalty0 (2):\penalty0 165--192,
  2002.
\newblock \doi{10.1137/S089548019936223X}.

\bibitem[Harris et~al.(2024)Harris, Kretschmann, and
  Mart\'inez~Mori]{harris2023lucky}
P.~E. Harris, J.~Kretschmann, and J.~C. Mart\'inez~Mori.
\newblock Lucky cars and the {\tt {q}uicksort} algorithm.
\newblock \emph{Amer. Math. Monthly}, 131\penalty0 (5):\penalty0 417--423,
  2024.
\newblock \doi{10.1080/00029890.2024.2309103}.

\bibitem[Knuth(1998)]{knuth1998linear}
D.~E. Knuth.
\newblock Linear probing and graphs.
\newblock \emph{Algorithmica}, 22\penalty0 (4):\penalty0 561--568, 1998.
\newblock \doi{10.1007/PL00009240}.

\bibitem[Konheim and Weiss(1966)]{konheim1966occupancy}
A.~G. Konheim and B.~Weiss.
\newblock An occupancy discipline and applications.
\newblock \emph{SIAM Journal on Applied Mathematics}, 14\penalty0 (6):\penalty0
  1266--1274, 1966.

\bibitem[Kreweras(1980)]{kreweras1980famille}
G.~Kreweras.
\newblock Une famille de polyn{\^o}mes ayant plusieurs propri{\'e}t{\'e}s
  {\'e}numeratives.
\newblock \emph{Periodica Mathematica Hungarica}, 11\penalty0 (4):\penalty0
  309--320, 1980.

\bibitem[Mart\'inez~Mori(2024)]{martinezmori2024what}
J.~C. Mart\'inez~Mori.
\newblock What is...a parking function?
\newblock \emph{Notices Amer. Math. Soc.}, 71\penalty0 (8):\penalty0
  1062--1065, 2024.

\bibitem[Maschler et~al.(1979)Maschler, Peleg, and
  Shapley]{maschler1979geometric}
M.~Maschler, B.~Peleg, and L.~S. Shapley.
\newblock Geometric properties of the kernel, nucleolus, and related solution
  concepts.
\newblock \emph{Math. Oper. Res.}, 4\penalty0 (4):\penalty0 303--338, 1979.
\newblock \doi{10.1287/moor.4.4.303}.

\bibitem[Myerson(1991)]{myerson1991game}
R.~B. Myerson.
\newblock \emph{Game theory: Analysis of conflict}.
\newblock Harvard University Press, 1991.

\bibitem[Peleg and Sudh\"olter(2007)]{peleg2007introduction}
B.~Peleg and P.~Sudh\"olter.
\newblock \emph{Introduction to the theory of cooperative games}, volume~34 of
  \emph{Theory and Decision Library. Series C: Game Theory, Mathematical
  Programming and Operations Research}.
\newblock Springer, Berlin, second edition, 2007.
\newblock ISBN 978-3-540-72944-0.

\bibitem[Pyke(1959)]{pyke1959supremum}
R.~Pyke.
\newblock The supremum and infimum of the {P}oisson process.
\newblock \emph{Ann. Math. Statist.}, 30:\penalty0 568--576, 1959.
\newblock \doi{10.1214/aoms/1177706269}.

\bibitem[Queyranne(1993)]{queyranne1993structure}
M.~Queyranne.
\newblock Structure of a simple scheduling polyhedron.
\newblock \emph{Math. Programming}, 58\penalty0 (2):\penalty0 263--285, 1993.
\newblock \doi{10.1007/BF01581271}.

\bibitem[Riordan(1969)]{riordan1969ballots}
J.~Riordan.
\newblock Ballots and trees.
\newblock \emph{J. Combinatorial Theory}, 6:\penalty0 408--411, 1969.

\bibitem[Roth(2016)]{roth2016lloyd}
A.~E. Roth.
\newblock Lloyd shapley (1923--2016).
\newblock \emph{Nature}, 532\penalty0 (7598):\penalty0 178--178, 2016.

\bibitem[Schulz and Uhan(2010)]{schulz2010sharing}
A.~S. Schulz and N.~A. Uhan.
\newblock Sharing supermodular costs.
\newblock \emph{Oper. Res.}, 58\penalty0 (4):\penalty0 1051--1056, 2010.
\newblock \doi{10.1287/opre.1100.0841}.

\bibitem[Schulz and Uhan(2013)]{schulz2013approximating}
A.~S. Schulz and N.~A. Uhan.
\newblock Approximating the least core value and least core of cooperative
  games with supermodular costs.
\newblock \emph{Discrete Optim.}, 10\penalty0 (2):\penalty0 163--180, 2013.
\newblock \doi{10.1016/j.disopt.2013.02.002}.

\bibitem[Shapley(1953)]{shapley1953value}
L.~S. Shapley.
\newblock A value for $n$-person games.
\newblock In H.~W. Kuhn and A.~W. Tucker, editors, \emph{Contributions to the
  Theory of Games, Volume II}, pages 307--318. Princeton University Press,
  Princeton, NJ, 1953.

\bibitem[Shapley(1955)]{shapley19555markets}
L.~S. Shapley.
\newblock Markets as cooperative games.
\newblock Technical report, RAND Corporation, Santa Monica, CA, 1955.

\bibitem[Shapley(1967)]{shapley1967balanced}
L.~S. Shapley.
\newblock On balanced sets and cores.
\newblock \emph{Naval Research Logistics Quarterly}, 14\penalty0 (4):\penalty0
  453--460, 1967.

\bibitem[Shapley(1971/72)]{shapley1971cores}
L.~S. Shapley.
\newblock Cores of convex games.
\newblock \emph{Internat. J. Game Theory}, 1:\penalty0 11--26; errata, ibid. 1
  (1971/72), 199, 1971/72.
\newblock \doi{10.1007/BF01753431}.

\bibitem[Shapley and Shubik(1966)]{shapley1966quasi}
L.~S. Shapley and M.~Shubik.
\newblock Quasi-cores in a monetary economy with nonconvex preferences.
\newblock \emph{Econometrica: Journal of the Econometric Society}, pages
  805--827, 1966.

\bibitem[Stanley(1996)]{stanley1996hyperplane}
R.~P. Stanley.
\newblock Hyperplane arrangements, interval orders, and trees.
\newblock \emph{Proc. Nat. Acad. Sci. U.S.A.}, 93\penalty0 (6):\penalty0
  2620--2625, 1996.
\newblock \doi{10.1073/pnas.93.6.2620}.

\bibitem[Stanley(1997)]{stanley1997parking}
R.~P. Stanley.
\newblock Parking functions and noncrossing partitions.
\newblock \emph{Electron. J. Combin.}, 4 (The Wilf Festschrift volume)\penalty0
  (2):\penalty0 Research Paper 20, 1997.
\newblock \doi{10.37236/1335}.

\bibitem[Stanley and Pitman(2002)]{stanley2002polytope}
R.~P. Stanley and J.~Pitman.
\newblock A polytope related to empirical distributions, plane trees, parking
  functions, and the associahedron.
\newblock \emph{Discrete Comput. Geom.}, 27\penalty0 (4):\penalty0 603--634,
  2002.
\newblock \doi{10.1007/s00454-002-2776-6}.

\bibitem[Yan(2015)]{yan2015parking}
C.~H. Yan.
\newblock Parking functions.
\newblock In M.~B\'ona, editor, \emph{Handbook of Enumerative Combinatorics},
  pages 835--894. CRC Press, Boca Raton, FL, 2015.

\end{thebibliography}
\label{sec:biblio}

\end{document}